\documentclass[11pt,reqno]{amsart}

\usepackage[T1]{fontenc}
\usepackage{lmodern}
\usepackage{microtype}
\usepackage[margin=1.08in]{geometry}
\usepackage{amsmath,amssymb,amsthm,mathtools}
\usepackage{enumitem}
\usepackage{xcolor}
\usepackage{hyperref}
\usepackage[nameinlink,capitalise,noabbrev]{cleveref}

\definecolor{reviewviolet}{RGB}{120,38,165}

\hypersetup{
  colorlinks=true,
  linkcolor=blue!45!black,
  citecolor=blue!45!black,
  urlcolor=blue!55!black,
  pdftitle={Bounded Killing fields and instability of Ricci-flat four-manifolds},
  pdfauthor={Tristan Ozuch}
}

\allowdisplaybreaks[2]
\setlist{itemsep=2pt,topsep=5pt}

\newtheorem{theorem}{Theorem}[section]
\newtheorem{quest}[theorem]{Question}
\newtheorem{proposition}[theorem]{Proposition}
\newtheorem{lemma}[theorem]{Lemma}
\newtheorem{corollary}[theorem]{Corollary}
\theoremstyle{definition}

\theoremstyle{remark}
\newtheorem{remark}[theorem]{Remark}

\newcommand{\R}{\mathbb R}
\newcommand{\Z}{\mathbb Z}

\newcommand{\cE}{\mathcal E}
\newcommand{\cQ}{\mathcal Q}
\newcommand{\Rm}{\operatorname{Rm}}
\newcommand{\Ric}{\operatorname{Ric}}
\newcommand{\tr}{\operatorname{tr}}

\title{Stability of gravitational instantons with a bounded Killing vector field}

\author{Tristan Ozuch}
\date{}

\begin{document}

\begin{abstract}
We prove that a complete ALF Ricci-flat 
$4$-manifold carrying a bounded Killing vector field is linearly stable if and only if it is locally hyperkähler. This applies uniformly to all known examples and in particular proves the instability of all of the metrics recently found by Li-Sun. The argument is based on infinitesimal Einstein-Maxwell deformations of the Ricci-flat metric associated with the Killing field.  

A separate simpler identity proves the instability of nonflat static Ricci-flat 4-metrics which include the static axisymmetric smooth Riemannian Myers/Korotkin-Nicolai metrics. 

To our knowledge, this proves the equivalence between linear stability and special holonomy and (anti-)selfduality for all known families of Ricci-flat $4$-manifolds. 

In higher dimensions, for generalizations of AF manifolds, we show that stability forces the universal cover to split a line. This gives explicit destabilizing tensors on the Riemannian Myers-Perry instantons and the Riemannian Schwarzschild-Tangherlini metrics in all dimensions.
\end{abstract}

\maketitle

\section{Introduction}
\label{sec:introduction}
Gravitational instantons were introduced by Hawking in Euclidean quantum gravity as nonsingular Riemannian solutions of the vacuum Einstein equations (i.e. Ricci-flat) in dimension $4$ and higher whose curvature decays at infinity or lies in $L^2$, by analogy with Yang-Mills instantons \cite{HawkingGI}. Their stability is a natural geometric, dynamical and physical question. This includes the Euclidean Schwarzschild instability in \cite{GrossPerryYaffe}. The stability of Riemannian metrics also governs the stability of some Lorentzian metrics as studied in \cite{GibbonsHartnoll,GibbonsHartnollPope}. In a number of works mentioned below, nonlinear (un)stable behaviors are deduced from the linear stability problem, which we focus on in this article. It is defined in the next section.

Ricci-flat metrics with a parallel spinor are semistable \cite{DaiWangWei} (in an adapted sense in the noncompact setting), and this is also true for locally hyperkähler metrics. Motivated by the relation between parallel spinors, special holonomy, and supersymmetric compactifications, Acharya conjectured that every stable Ricci-flat metric on a compact simply connected manifold has special holonomy \cite[Conjecture~2]{Acharya}. The same question for noncompact Ricci-flat metrics is central. 
\begin{quest}
    Must a semistable gravitational instanton $4$-manifold be locally hyperkähler?
\end{quest}
The main \cref{thm:intro-bounded} below gives a partial answer: for complete ALF Ricci-flat metrics carrying a bounded Killing field, which is satisfied by all known examples, stability forces a local hyperkähler holonomy reduction. 
\\

The hyperkähler subclass provides many of the known examples of gravitational instantons. The Gibbons-Hawking ansatz gives the fundamental circle-invariant multi-center metrics \cite{GibbonsHawkingMulti}, while Kronheimer classified the ALE hyperkähler spaces \cite{KronheimerALE,KronheimerTorelli}. Minerbe proved that hyperkähler gravitational instantons with cubic volume growth are ALF \cite{MinerbeAsymptotic} and that the cyclic ALF spaces are precisely the flat product and the multi-Taub-NUT metrics \cite{MinerbeMultiTN} while Chen-Chen completed the corresponding classification in the dihedral case \cite{ChenChenII}. More generally, Sun-Zhang \cite[Theorem~1.2]{SunZhang} proved that every nonflat finite-energy hyperkähler four-manifold is asymptotic to one of the six model types: ALE, ALF, ALG, ALG$^*$, ALH, ALH$^*$.

The Chen-Teo family, which disproved the classical conjecture that every nonflat AF gravitational instanton belongs to the Riemannian Kerr family \cite{ChenTeo}, became a catalyst for the recent study of gravitational instantons after Aksteiner and Andersson recognized it as Hermitian and conformally Kähler \cite{AksteinerAndersson}.  Biquard-Gauduchon classified toric Hermitian ALF instantons \cite{BiquardGauduchon}, M.~Li gave a classification of all Hermitian gravitational instantons \cite{MingyangLi}, and Andersson-Araneda established their infinitesimal rigidity \cite{AnderssonAraneda} based on a result of Biquard-Gauduchon-LeBrun \cite{BiquardGauduchonLeBrun}. Aksteiner-Andersson-Dahl-Nilsson-Simon additionally established global rigidity for gravitational instantons with a circle symmetry \cite{AADSNS}: they must be Hermitian in a number of situations. O. Biquard and the author proved that Hermitian gravitational instantons were unstable in \cite{BO} using Delay's conformal Laplacian from \cite{DelayConformalLaplacian}.
 
 Beyond the Hermitian world, in a recent breakthrough, Li-Sun constructed a remarkable family of complete toric $\mathrm{AF}_\beta$ metrics with arbitrary second Betti number, whose members with $b_2\geqslant3$ are not locally Hermitian in either orientation \cite{LS}.  Another lesser known family of metrics is given by the Riemannian versions of the static axisymmetric metrics of Myers and Korotkin-Nicolai, \cite{Myers,KorotkinNicolai}. These new examples motivate the search for a stability criterion that treats non-Hermitian instantons and this is the purpose of this article.

\subsection{Stability of gravitational instantons}

Let $(M^4,g)$ be an oriented Ricci-flat manifold.  On symmetric two-tensors we work with the following Lichnerowicz operator
\begin{equation}\label{eq:E-intro}
 \cE_g h=\nabla^*\nabla h-2\mathring{\Rm}_g h
\end{equation}
where for a symmetric two-tensor $h$, we use $(\mathring{\Rm}h)_{ij}=R_{ikjl}h^{kl}$, and its quadratic form
\begin{equation}\label{eq:Q-intro}
 \cQ_g(h)=\int_M\langle\cE_gh,h\rangle\,dv_g
 =\int_M\bigl(|\nabla h|^2-2\langle\mathring{\Rm}_gh,h\rangle\bigr)\,dv_g.
\end{equation}
A symmetric tensor is \emph{transverse-traceless} (TT) if
$\tr_gh=0=\delta_gh$.  On a noncompact manifold we call $g$
\emph{(linearly) unstable} if there is a smooth TT tensor $h$ such that $\cQ_g(h)$ is finite and negative.  This is at the same time: 
\begin{itemize}
    \item the standard variational instability for the Yamabe problem,
\item the variational notion underlying dynamical stability and instability
under the Ricci flow (see
\cite{GuentherIsenbergKnopf,Sesum,Haslhofer,HaslhoferMuller}
for compact metrics,
\cite{DeruelleKroncke,DeruelleOzuchLojasiewicz,
DeruelleOzuchDynamical,KronckePetersen,StolarskiWaldron}
for ALE metrics, and
\cite{KO}
for ALF metrics; related ALE/AE applications appear in
\cite{DeruelleOzuchOrbifold,LopezOzuch,
BaldaufOzuchMass,KronckePetersen,BaldaufOzuchEnergy}), and

    \item the condition that allows perturbations increasing scalar curvature, whose global consequences include
mass-decreasing scalar-flat perturbations \cite{DahlKroncke}.
\end{itemize}
 Thanks to these nonlinear instability consequences, we will often loosely use ``unstable'' in place of ``linearly unstable.'' In the present paper, assuming that a Ricci-flat ALF metric admits a bounded Killing vector field, we show that it is stable if and only if it is locally hyperkähler.

\subsection{Stability and holonomy reduction with a bounded Killing vector field}

In this article, we work within the broad ALF definition of \cite[Definition~1.1]{BiquardGauduchon}, which include in particular the $\mathrm{AF}_\beta$ end of \cite[Definition~4.30]{LS} for all $\beta\in \mathbb R$. In the notation of Biquard-Gauduchon \cite{BiquardGauduchon}, the general model is $$b=dr^2+r^2\gamma+\eta^2,$$ where $\eta$ represent the $1$-form associated with the unit-length fiber at infinity and $\gamma$ is a round $2$-metric. Still in their notations, $T$ is the asymptotic Killing vector field satisfying $\eta(T)=1$ and $ \iota_Td\eta=0$. The precise definition of $(L,\eta,T,\gamma)$ are recalled in \cref{sec:end-analysis}. We will consider metrics with the following condition satisfied by all known examples:
\begin{equation}\label{eq:intro-end}
 |\nabla_b^j(g-b)|_b=\mathcal{O}(r^{-1-j})\qquad \text{ for } \quad j\geqslant0.
\end{equation}

\begin{theorem}\label{thm:intro-bounded}
Let $(M^4,g)$ be complete, oriented ALF Ricci-flat $4$-manifold. Suppose that $g$ carries a nonzero bounded Killing field.  Then we have the following equivalence
\begin{equation}\label{eq:intro-dichotomy}
 g\text{ is linearly stable} \iff W_g^+=0\text{ or }W_g^-=0.
\end{equation}
\end{theorem}

This directly applies to the new metrics found by Li-Sun.

\begin{corollary}\label{cor:intro-LS}
All Li-Sun metrics \cite{LS} are unstable.
\end{corollary}

\vspace{0.3cm}

In higher dimensions, stability with an asymptotically parallel Killing vector field instead forces the splitting of the metric on its universal cover. We in particular recover the instability of \cite{DelaySchwarzschildTangherlini} and prove new ones. For the generalized AF ends considered here, we take
\[
 b=g_{\mathbb R^{n-1}}+d\theta^2,
 \qquad T=\partial_\theta.
\]
\begin{remark}
    Our applications more generally apply to asymptotically flat ALC ends corresponding to replacing $g_{\mathbb R^{n-1}}$ by any $(n-1)$-dimensional Ricci-flat cone, as long as the Fredholm analysis justifies specific integrations by parts.  
\end{remark}
We assume that our spaces have a \textit{generalized AF} end with an asymptotically constant Killing vector field $K$ satisfying:
\begin{equation}\label{eq:higher-ALF-decay}
 |\nabla_b^j(g-b)|_b+|\nabla_b^j(K-T)|_b
 =\mathcal O(r^{-\tau-j}),\qquad \text{ for }
 \tau>\frac{n-3}{2},\text{ and } j\geqslant0.
\end{equation}
This is the higher-dimensional analogue of an AF end in gravitational instanton terminology, its volume growth is $r^{n-1}$. 

\begin{remark}
For $n\geqslant5$, the metric decay in \eqref{eq:higher-ALF-decay} is the AF class of \cite[Definition~1.1]{KhuriWangMass}.  Mass questions for these AF and ALF ends are studied in \cite{MinerbeMass,KhuriWangMass}, and the relative mass with respect to a Ricci-flat reference introduced in \cite{KO} is the mass difference used in the toric comparison theorem of \cite{AlaeeKhuriKunduriMass}. 
\end{remark}

\begin{theorem}
\label{thm:higher-dimensional-alf}
Let $n\geqslant5$, and let $(M^n,g)$ be a complete Ricci-flat manifold
with a generalized AF end.  If a Killing field $K$ and the end satisfy
\eqref{eq:higher-ALF-decay}, then $g$ is linearly unstable unless $K$ is
parallel.
\end{theorem}

\begin{remark}
    The proof in particular yield explicit destabilizing tensors for the Riemannian Myers-Perry metrics of \eqref{eq:euclidean-myers-perry} and all of the Riemannian Schwarzschild-Tangherlini metrics of \eqref{eq: STeuc} in arbitrary dimensions, see \cref{sec:myers-perry}. In that case our result recovers, in every dimension, the instability proved by a different construction for $4\leqslant n\leqslant9$ in \cite{DelaySchwarzschildTangherlini}. 
\end{remark}

Combining \cref{thm:higher-dimensional-alf} with parallel-spinor
semistability gives a particularly simple rigidity statement: every such
special-holonomy metric is flat. 
\begin{corollary}
\label{cor:higher-special-holonomy}
Under the hypotheses of \cref{thm:higher-dimensional-alf-bis}, if $g$, or a finite cover of $g$, carries a nonzero parallel spinor, then $g$ is flat.
\end{corollary}

\subsection{Idea of proof for the $4$-dimensional result }
We fix notation once for the entire paper.  For a Killing field $K$ we write
\begin{equation}\label{eq:intro-data}
 \kappa=K^\flat,\qquad F=d\kappa,\quad  \text{ and }\quad
 F^\pm=\tfrac12(F\pm*F).
\end{equation}
Two-forms are normalized by $|\alpha|^2=\frac12\alpha_{ij}\alpha^{ij}$, and the composition of two-tensors is $(\beta\circ\alpha)_{ij}=\beta_i{}^p\alpha_{jp}$ to match Maxwell theory conventions.  The scalar Laplacian $\Delta=-\tr\nabla^2$ is the positive one.  With these choices, LeBrun noted in \cite{LeBrunEM}, that the standard Maxwell stress reads
\begin{equation}\label{eq:intro-Maxwell-algebra}
 F_{ip}F_j{}^p-\tfrac14F_{pq}F^{pq}g_{ij}=2(F^+\circ F^-)_{ij}.
\end{equation}
We also use the following Riemannian curvature along the article:
\[
 R(X,Y)Z=\nabla_X\nabla_YZ-\nabla_Y\nabla_XZ-\nabla_{[X,Y]}Z,
 \qquad
 \Rm(X,Y,Z,W)=\langle R(X,Y)W,Z\rangle .
\]
Thus the sectional curvatures are $\Rm(X,Y,X,Y)$, and
$R_{ijkl}=\Rm(e_i,e_j,e_k,e_l)$.

Consider a Ricci-flat
four-manifold $(M^4,g)$, a Killing field $K$ and assume, up to rescaling, that its asymptotic length is $1$. We \emph{complete} the main geometric perturbation $-\kappa\otimes\kappa$ thanks to a Hessian and a pure trace term to construct a TT tensor, named $k_K$ solving the equation
\begin{equation}\label{eq:intro-inverse}
 \cE_gk_K=2F^+\circ F^-.
\end{equation}
The intuition that lead us to consider this deformation is that it is an \textit{infinitesimal Einstein-Maxwell deformation} of $g$: by LeBrun \cite{LeBrunEM}, the Riemannian Einstein-Maxwell equation in
dimension four reads $\Ric=2F^+\circ F^-$ for a harmonic two-form $F=F^++F^-$. A maximum principle argument , computations and integrations by parts lead to the sign for the stability quadratic form:
\begin{align}
 \cQ_g(k_K)  \leqslant 0,
 \label{eq:intro-sign-plus}
\end{align} with equality forcing $F$ to be either self-dual or anti-self-dual, which in turn forces the metric to satisfy $W^+=0$ or $W^-=0$ by Kostant's formula.

\begin{remark}
    The hypothesis that the Killing vector field is bounded is crucial and essentially forces the metric to be ALF:
    \begin{itemize}
        \item on an ALE end a Killing field grows linearly and while the formulae still hold, the resulting tensor $k_K$ is not admissible: $u$ grows quartically and $k_K$ quadratically. Additionally, the equation is typically not solvable since $2F^+\circ F^-$ is not orthogonal to the $L^2$-kernel of Ricci-flat ALE metrics, see  \cref{sec:ALE-remark}.
        \item bounded Killing vector fields cannot exist on the typical other typical ends of Ricci-flat metrics, namely ALG, ALG$^*$, Kasner, ALH or ALH$^*$, see \cref{sec:other-ends}.
    \end{itemize}
\end{remark} 

\subsection{Linear instability of nonflat static Ricci-flat metrics}

For static Ricci-flat metrics, the destabilizing tensor is even simpler. 

\begin{theorem}\label{thm:intro-static}
Let $(M^4,g)$ be a nonflat, connected Ricci-flat four-manifold carrying a nonzero Killing field $K$ with $\kappa\wedge d\kappa=0$. Then
\begin{equation}\label{eq:intro-static-formula}
 h=F^+\circ F^-\text{ is TT},\qquad \text{ and } \qquad
 \langle\cE_gh,h\rangle=-\tfrac12\left|d\Big(\tfrac{|F|^2}{2}\Big)\right|^2\leqslant 0,
\end{equation}
with equality identically only if $g$ is flat.
This implies the instability of every smooth Riemannian Myers-Korotkin-Nicolai metric \cite{Myers,KorotkinNicolai,ReirisPeraza}.
\end{theorem}

\section*{Acknowledgements}
The author thanks Olivier Biquard for suggesting to examine the construction in higher dimensions.  This work is supported by the National Science Foundation under Grant No.~DMS-2405328.

\section{Destabilization of static Ricci-flat manifolds}
\label{sec:harmonic-static}

\subsection{A natural family of test tensors on Einstein $4$-manifolds}
The convention $(\beta\circ\alpha)_{ij}=\beta_i{}^p\alpha_{jp}$
gives an isometry
\begin{align*}
    \Lambda^+\otimes\Lambda^-&\longrightarrow S^2_0T^*M,\\
  \alpha^+\otimes\beta^-&\longmapsto\alpha^+\circ\beta^- ,
\end{align*}
whose image is the space of trace-free symmetric $2$-tensors.

\begin{proposition}\label{prop:harmonic-product}
Let $(M^4,g)$ be oriented Einstein with scalar curvature $s$.  If
$\alpha^+\in\Omega^2_+(M)$ and $\beta^-\in\Omega^2_-(M)$ are harmonic,
then $h=\alpha^+\circ\beta^-$ is TT, and
\begin{equation}
    \cE_g h=-\frac{s}{2}h-2\sum_a(\nabla_a\alpha^+)\circ(\nabla_a\beta^-),
 \label{eq:EH-harmonic}
\end{equation}
which implies from $\big\langle(\nabla_a\alpha^+)\circ(\nabla_a\beta^-),
 \alpha^+\circ\beta^-\big\rangle
 =\langle\nabla_a\alpha^+,\alpha^+\rangle
  \langle\nabla_a\beta^-,\beta^-\rangle,$ that
\begin{equation}
    \langle\cE_g h,h\rangle=-\frac{s}{2}|h|^2-\frac12\langle d|\alpha^+|^2,d|\beta^-|^2\rangle.\label{eq:energy-harmonic}
\end{equation}
\end{proposition}
\begin{remark}
    Consider an Einstein $4$-manifold with positive scalar curvature, the above lemma leads to a large class of test TT symmetric $2$-tensors. The stability of the metric would imply an ``anticorrelation'' of the norms of selfdual and anti-selfdual $2$-forms. 
\end{remark}
\begin{proof}
Consider $\alpha^+\in\Omega^2_+(M)$ and $\beta^-\in\Omega^2_-(M)$ closed and
coclosed $\pm$-selfdual $2$-forms on $M$.  The tensor $h:=\beta^-\circ\alpha^+$ is trace-free algebraically and it is divergence-free by \cite[Eq.~(6)]{LeBrunEM}. 

For an Einstein four-metric with scalar curvature $s$, the Hodge Weitzenb\"ock formulas are
\[
 \nabla^*\nabla\alpha^+=2W^+\alpha^+-\frac{s}{3}\alpha^+,
 \qquad \text{ and }\qquad
 \nabla^*\nabla\beta^-=2W^-\beta^- -\frac{s}{3}\beta^-.
\]
The product rule for $\nabla^*\nabla(\alpha^+\circ\beta^-)$ and
\[
 \mathring{\Rm}(\beta^-\circ\alpha^+)
 =(W^-\beta^-)\circ\alpha^++\beta^-\circ(W^+\alpha^+)-\frac{s}{12}h.
\]
gives \eqref{eq:EH-harmonic} on an Einstein metric, and \eqref{eq:energy-harmonic} follows directly. See \cite{NO1,NO2} for similar computations that additionally apply to Ricci solitons. 
\end{proof}

\subsection{Application to static Ricci-flat manifolds}

For a Killing field $K$, denote $\kappa:=K^\flat$, then $\delta \kappa=0$ and the Killing identity gives
\begin{equation}\label{eq:killing}
 \delta d \kappa=2\Ric(K,\cdot).
\end{equation}
Thus, the $2$-form
\begin{equation}\label{eq:Pap-static}
 F=d\kappa
\end{equation}
is closed and coclosed when $g$ is Ricci-flat.

We are now ready to prove \cref{thm:intro-static} in the following refined form. 
\begin{theorem}[Refined \cref{thm:intro-static}]
Let $(M^4,g)$ be connected, oriented and Ricci-flat, and let $K\not\equiv0$ be a
Killing field satisfying
\begin{equation}\label{eq:intro-static}
 \kappa\wedge d\kappa=0.
\end{equation}
Put $F=d\kappa$, $F^\pm$ its $\pm$-selfdual parts, and $h=F^-\circ F^+$.  Then $h$ is TT and
\begin{equation}\label{eq:intro-static-energy}
 \langle\cE_gh,h\rangle
 =-\frac12\Big|d\Big(\tfrac{|F|^2}{2}\Big)\Big|^2.
\end{equation}
If the right-hand side vanishes identically, then $g$ is flat. 
\end{theorem}

\begin{proof}
Equation \eqref{eq:intro-static} gives
\[
 F\wedge F=d(\kappa\wedge F)=d(\kappa\wedge d \kappa)=0,
\]
Thus $|F^+|^2=|F^-|^2=q$, where $q=|F|^2/2$.
Applying \cref{prop:harmonic-product} to $F^-\circ F^+$ proves \eqref{eq:intro-static-energy}.

For the equality case, recall that on any connected open set
where $K\neq0$ the condition $\kappa\wedge d\kappa=0$ means that
$K^\perp$ is an integrable distribution and produces coordinates in which
\begin{equation}\label{eq:static-splitting}
 g=N^2d\tau^2+\bar g,\qquad K=\partial_\tau,\qquad N=|K|,\qquad
 \partial_\tau N=0,\quad\mathcal L_K\bar g=0,
\end{equation}
where $\bar g$ is a Riemannian metric on the orthogonal $3$-manifold, and $F=2\,dN\wedge(N\,d\tau)$. 

Ricci-flatness becomes equivalent to the static Einstein equations
\begin{equation}\label{eq:static-vacuum}
 \bar\Delta N=0,\qquad
 \Ric_{\bar g}=N^{-1}\bar\nabla^2N,\qquad
 q=2|\bar\nabla N|^2.
\end{equation}
If $dq=0$, then
$\bar\nabla^2N(\bar\nabla N,\cdot)=0$.  The Bochner identity gives
\[
 0=-\frac12\bar\Delta|\bar\nabla N|^2
  =|\bar\nabla^2N|^2
   +N^{-1}\bar\nabla^2N(\bar\nabla N,\bar\nabla N)
  =|\bar\nabla^2N|^2.
\]
Hence $\bar g$ is Ricci-flat and therefore flat in dimension three, and $g$ is flat on $\{K\neq0\}$. Since a nonzero Killing field cannot vanish on an open set, $g$ is flat everywhere.
\end{proof}

\subsection{The periodic Schwarzschild family}

The smooth Riemannian versions of the Myers/Korotkin-Nicolai metrics
\cite{Myers,KorotkinNicolai,ReirisPeraza} depend, up to
homothety, on a single parameter $a\in(0,2)$, and has Weyl form
\begin{equation}\label{eq:MKN-metric}
 g=e^{\mathcal U}d\tau^2+e^{-\mathcal U}
  \bigl[e^{2\mathcal K}(d\rho^2+dz^2)+\rho^2d\phi^2\bigr],
 \qquad z\in\R/\ell\Z, \quad\ell>0
\end{equation}
and is static with $K=\partial_\tau$, and the period of $K$ is chosen to make the metric smooth and that of $z$ determines the topology (it is infinite, periodic if $z\in\mathbb R$). For these asymptotics, see \cite[Section 1]{ReirisPeraza} and for the original Weyl reduction, see
\cite{Weyl}. At infinity, the asymptotics are
\begin{equation}\label{eq:MKN-asymptotics}
 \mathcal U=a\log\rho+c_0+\mathcal{O}(e^{-c\rho}),\qquad \text{ and }\qquad
 \mathcal K=\frac{a^2}{4}\log\rho+c_1+\mathcal{O}(e^{-c\rho}).
\end{equation}
It is convenient to introduce a radial distance parameter
\begin{equation}\label{eq:MKN-radius}
 R =\rho^\gamma,\qquad
 \gamma=\frac{a^2-2a+4}{4},\qquad \text{ and } \qquad
 dv_g\approx c R\,dR\,d\tau\,dz\,d\phi.
\end{equation}
For the tensor in \cref{thm:intro-static}, we explicitly find for $\delta=\frac{2(2-a)^2}{a^2-2a+4}>0$,
\begin{equation}\label{eq:MKN-decay}
 |\nabla^jh|=\mathcal{O}(R^{-\delta-j}),\qquad \text{ for }0\leqslant j\leqslant2.
\end{equation}
The volume growth in \eqref{eq:MKN-radius} makes the energy and cutoff error finite for every $\delta>0$.

\begin{remark}
    The asymptotic Kasner
exponents are
\begin{equation}\label{eq:MKN-Kasner}
 (p_\tau,p_z,p_\phi)
 =\frac1{a^2-2a+4}\bigl(2a,\,a^2-2a,\,2(2-a)\bigr).
\end{equation}
 As explained before \cite[Theorem~3 and Remark~4.1]{KRWY}, the metric constructed there corresponds to the Wick rotation of the space-periodic Myers/Korotkin-Nicolai solution with ``equal rod length'' assumed for simplicity. When the rods have different lengths, one directly checks $a\neq1$: the corresponding smooth members mentioned in \cite[Remark~4.1]{KRWY} supply the ends of with curvature of a generic algebraic type asked for in \cite[Question~6 and Section 6.8]{LiSunSurvey}.
\end{remark}

\section{Tensors associated with a Killing field}
\label{sec:response}

Let $K$ be Killing on a Ricci-flat four-manifold and recall the notations
\begin{equation}\label{eq:Pap-data}
 \kappa=K^\flat,\qquad V=|K|^2,\qquad F=d\kappa,\qquad
 F^\pm=\tfrac12(F\pm*F),
\end{equation}
and our conventions: $\Delta=d^*d=-\tr\nabla^2$ and
$(\delta h)_j=-\nabla^ih_{ij}$.

\begin{lemma}\label{lem:Killing-identities}
On a Ricci-flat four-manifold, we have
\begin{align}
 dF=0,\qquad\delta F&=0,\quad \text{ and } \quad\iota_KF=-dV,\label{eq:KM}\\
 \nabla^2V&=\frac12F\circ F-2\Rm(\,\cdot\,,K,\,\cdot\,,K),
 \label{eq:Hess-V}\\
 \Delta V&=-|F|^2,
 \label{eq:Delta-V}\\
 \cE_g(\kappa\otimes\kappa)&=\nabla^2V-F\circ F,
 \label{eq:E-kk}
\end{align}
and for every smooth function $f$, we find
\begin{equation}\label{eq:E-natural}
 \cE_g(\nabla^2f)=\nabla^2(\Delta f),\qquad
 \cE_g(fg)=(\Delta f)g.
\end{equation}
\end{lemma}

\begin{proof}

First, Cartan's formula and the definition $\iota_K\kappa=V$ give
$\iota_KF = \iota_K d\kappa =\mathcal L_K\kappa-d\iota_K\kappa=-dV$. Indeed, since $\kappa=g(K,\cdot)$,
\[
 \mathcal L_K\kappa=(\mathcal L_Kg)(K,\cdot)
                    +g(\mathcal L_KK,\cdot)=0,
\]
because $\mathcal L_Kg=0$ and $\mathcal L_KK=[K,K]=0$.

For a Killing field, the standard second-derivative identity
\cite[Theorem~1.81(f)]{Besse}, written with our curvature convention, is
\[
 (\nabla^2K)(Y,Z):=\nabla_Y\nabla_ZK-\nabla_{\nabla_YZ}K = R(Y,K)Z,
\]
hence, we get
 $\big\langle(\nabla^2K)(Y,Z),K\big\rangle
 =-\Rm(Y,K,Z,K)$ and consequently
\[
 \nabla^2V(Y,Z)
 =2\langle\nabla_YK,\nabla_ZK\rangle-2\Rm(Y,K,Z,K)
 =\frac12(F\circ F)(Y,Z)-2\Rm(Y,K,Z,K),
\]
% {\color{blue}
% Indeed, since $V=|K|^2$, for any vector field $Z$ one has
%  $dV(Z)=2\langle\nabla_ZK,K\rangle$ and consequently, we find
% \begin{align*}
%  \nabla^2V(Y,Z)
%  &=Y\bigl(dV(Z)\bigr)-dV(\nabla_YZ)\\
%  &=2\big\langle
%    \nabla_Y\nabla_ZK-\nabla_{\nabla_YZ}K,K
%    \big\rangle
%    +2\langle\nabla_ZK,\nabla_YK\rangle\\
%  &=2\langle(\nabla^2K)(Y,Z),K\rangle
%    +2\langle\nabla_YK,\nabla_ZK\rangle.
% \end{align*}
% }
which is \eqref{eq:Hess-V}.
Tracing \eqref{eq:Hess-V} in an orthonormal frame gives
\[
-\Delta V= \tr\nabla^2V
 =\frac12F_{pq}F^{pq}-2\sum_a\Rm(e_a,K,e_a,K)
 =|F|^2-2\Ric(K,K)=|F|^2.
\]

Since $\nabla^*\nabla\kappa=0$, the tensorial product rule and from the Killing equation $\nabla\kappa = \frac12\nabla_{\operatorname{Sym}}\kappa + \frac12 d\kappa = \frac12 d\kappa=\frac12F$, give
\[
 \nabla^*\nabla(\kappa\otimes\kappa)
 =-2\sum_a\nabla_{e_a}\kappa\otimes\nabla_{e_a}\kappa
 =-\frac12F\circ F.
\]
Moreover, we have
$\mathring{\Rm}(\kappa\otimes\kappa)=\Rm(\,\cdot\,,K,\,\cdot\,,K)$,
and hence by \eqref{eq:Hess-V}, we find
\[
 \cE_g(\kappa\otimes\kappa)
 =-\frac12F\circ F-2\Rm(\,\cdot\,,K,\,\cdot\,,K)
 =\nabla^2V-F\circ F.
\]

Finally, the Lichnerowicz commutation identities \cite{Lichnerowicz} imply \eqref{eq:E-natural}.
\end{proof}

\begin{theorem}\label{thm:scalar-response}
Let $c\in\R$.  If $u$ solves $\Delta u=V+2c$, then
\begin{equation}\label{eq:k-def}
 k_K=\nabla^2u-\kappa\otimes\kappa+\frac{V+c}{2}g
\end{equation}
is TT and
\begin{equation}\label{eq:E-k}
 \cE_gk_K=(F\circ F)_0=2F^+\circ F^-.
\end{equation}
\end{theorem}

\begin{proof}
By tracing and using $\tr\nabla^2u=-\Delta u$, we find
\[
 \tr k_K=-(V+2c)-V+2(V+c)=0.
\]
On a Ricci-flat metric,
$\delta\nabla^2u=d\Delta u=dV$.  Since $\delta\kappa=0$,
$\nabla\kappa=\frac12F$, and $\iota_KF=-dV$,
\[
 \delta(\kappa\otimes\kappa)
 =(\delta\kappa) \kappa-\nabla_K\kappa=-\nabla_K\kappa
 =-\frac12\iota_KF
 =\frac12dV,
 \qquad
 \delta\!\left(\frac{V+c}{2}g\right)=-\frac12dV.
\]
Thus $\delta k_K=0$.
Using \cref{lem:Killing-identities} and
$\Delta u=V+2c$,
\[
\begin{aligned}
 \cE_gk_K
 &=\nabla^2(\Delta u)-\bigl(\nabla^2V-F\circ F\bigr)
   +\frac12(\Delta V)g\\
 &=F\circ F-\frac12|F|^2g
  =(F\circ F)_0
  =2F^+\circ F^-,
\end{aligned}
\]
where the last equality is \eqref{eq:intro-Maxwell-algebra}.
\end{proof}

\begin{remark}\label{rem:local-identity}
The identity \eqref{eq:E-k} is entirely local: it requires only a
solution of the scalar Poisson equation $\Delta u=V+2c$. 
\end{remark}

\subsection{An Einstein-Maxwell interpretation}
\label{subsec:EM}

Equation \eqref{eq:E-k} is exactly the first term in a formal Einstein-Maxwell deformation.  With the normalization of \cite[Eqs.~(1)-(3)]{LeBrunEM},
\begin{equation}\label{eq:EM-jet}
 g_t=g+t^2k_K+\mathcal{O}(t^4),\qquad
 A_t=\tfrac{t}{\sqrt2}\kappa+\mathcal{O}(t^3)
\end{equation}
satisfies Einstein-Maxwell's equation to leading order with $F_t = dA_t$.  This matches nicely with the Einstein-Maxwell variations considered in \cite{Harrison,ErnstMagnetic} whose metric part
\[
 g_t =\Lambda_t^2(g-V^{-1}\kappa\otimes\kappa)
              +\Lambda_t^{-2}V^{-1}\kappa\otimes\kappa,
 \qquad \Lambda_t=1+\tfrac{t^2V}{4},
\]
has second-order coefficient
$\frac{d^2}{dt^2}\big|_{t=0}g_t = Vg-2\kappa\otimes\kappa$.  The remaining term $\nabla^2u+\tfrac c2g$ in \eqref{eq:k-def} is the diffeomorphism plus homothety needed to place the coefficient in TT gauge.  Analogous exact Riemannian Einstein-Maxwell families
include the Reissner-Taub-NUT metrics \cite{Page} and the recent examples of Araneda-Dunajski
\cite{AD}. 

\subsection{Comparison with Biquard-Ozuch deformation}

Suppose $g=f^2\widetilde g$ is Ricci-flat and
$(\widetilde g,J,\widetilde\omega)$ is K\"ahler, with
$f^{-1}=\operatorname{Scal}_{\widetilde g}/6$ as in \cite{BO}.
Their distinguished asymptotically constant Killing field is
\begin{equation}\label{eq:BO-X}
 X=J\nabla^gf.
\end{equation}
If $F_X=dX^\flat$, then a direct K\"ahler calculation gives
\begin{equation}\label{eq:BO-splitting}
 F_X^+=\frac12\widetilde\omega.
\end{equation}
Denote $h_0$ the TT projected tensor used in \cite{BO}, one checks that
\begin{equation}\label{eq:BO-identification}
\cE h_0=4F_X^+\circ F_X^-.
\end{equation}
Consequently $h_0/2$ and the tensor $k_X$ in
\cref{thm:scalar-response} solve the same equation.

\subsection{A higher-dimensional variant}
\label{subsec:higher-dimensional-variant}

Let $(M^n,g)$ be Ricci-flat, let $K$ be Killing, and set again
\[
 \kappa=K^\flat,\qquad V=|K|^2,\qquad F=d\kappa.
\]

Unlike in dimension $4$, the traceless symmetric $2$-tensor
$(F\circ F)_0$ is not divergence-free and the divergence-free Maxwell tensor $F\circ F-\frac12|F|^2g$ is not traceless, so we will instead target the TT
\[
 (F\circ F)_0-\frac{n-4}{2(n-1)}(\nabla^2V)_0.
\]
When $n=4$ it is equal to the TT tensor
$(F\circ F)_0=F\circ F-\frac12|F|^2g$.

Introduce the following numbers
\begin{equation}\label{eq:higher-coefficients}
 a_n=\frac{n+2}{2(n-1)},\qquad \text{ and }\qquad
 b_n=\frac{3}{2(n-1)}.
\end{equation}
Then, as in the previous section, we construct a TT preimage of the tensor $(F\circ F)_0-\frac{n-4}{2(n-1)}(\nabla^2V)_0$.

\begin{proposition}\label{prop:higher-dimensional-TT}
If $\Delta u=a_nV+\lambda$, then
\begin{equation}\label{eq:higher-k-def}
 k=\nabla^2u-\kappa\otimes\kappa
       +\left(b_nV+\frac{\lambda}{n}\right)g
\end{equation}
is TT and
\begin{equation}\label{eq:higher-Ek}
 \cE_gk=(F\circ F)_0
       -\frac{n-4}{2(n-1)}(\nabla^2V)_0.
\end{equation}
\end{proposition}

We can now prove \cref{thm:higher-dimensional-alf} by using $k$ as our test tensor.

\begin{theorem}[{\cref{thm:higher-dimensional-alf}}]
\label{thm:higher-dimensional-alf-bis}
Let $n\geqslant5$ and assume that the end and a Killing field
$K$ satisfy \eqref{eq:higher-ALF-decay}.  Then, stability forces $K$ to be parallel and the universal cover
to split a line:
\begin{equation}\label{eq:higher-splitting}
 (\widetilde M,\widetilde g)
 \approx(\mathbb R,dt^2)\times(N^{n-1},h).
\end{equation}
\end{theorem}

\begin{proof}
Choose $\lambda_n=n-1-a_n=\frac{n(2n-5)}{2(n-1)}$ and let $u$ and $k$ be given by
\cref{prop:higher-dimensional-TT,subsec:higher-ALF-analysis}.
The estimates there justify all the following integrations by parts
and make the boundary terms vanish.  Since $\Delta V=-|F|^2$ and
$V\to1$ fast enough by \cref{subsec:higher-ALF-analysis}, integrating by parts against $V$ yields
\begin{equation}\label{eq:higher-flux}
 \int_M|F|^2\,dv_g
 =\int_M|dV|^2\,dv_g+\int_MV|F|^2\,dv_g.
\end{equation}

The only difference from the 4D pairing is the Hessian term in \eqref{eq:higher-Ek}, but since $k$ is TT, we find
\[
 \int_M\langle(\nabla^2V)_0,k\rangle\,dv_g
 =\lim_{R\to\infty}\int_{\partial M_R}
       k(\nabla V,\nu)\,dA=0.
\]
Consequently, we finally find
\[
 \cQ_g(k)=\int_M\langle F\circ F,k\rangle\,dv_g.
\]
Now, from
 $\delta(F\circ F)=-\frac12d|F|^2,
 $ $
 (F\circ F)(K,K)=|dV|^2$ and $
 \tr_g(F\circ F)=2|F|^2,$ we deduce
\[
 \int_M\langle F\circ F,\nabla^2u\rangle\,dv_g
 =-\frac12\int_M(a_nV+\lambda_n)|F|^2\,dv_g.
\]
Combined with \eqref{eq:higher-k-def}, we therefore obtain
\begin{align}
 \cQ_g(k)
 &=-\int_M|dV|^2\,dv_g
   +\left(2b_n-\frac{a_n}{2}\right)
      \int_MV|F|^2\,dv_g \notag\\
 &\quad
   +\left(\frac{2\lambda_n}{n}-\frac{\lambda_n}{2}\right)
      \int_M|F|^2\,dv_g \notag\\
 &=-\frac{3(n+2)}{4(n-1)}
      \int_M|dV|^2\,dv_g
   -\frac{n-5}{2}\int_M|F|^2\,dv_g.
 \label{eq:higher-energy-signed}
\end{align}
where the last equality uses \eqref{eq:higher-flux} and
\eqref{eq:higher-coefficients}.  For $n\geqslant5$, stability thus
forces $dV=0$.  The identity $\Delta V=-|F|^2$ then gives $F=0$, so
$K$ is parallel.
\end{proof}

\begin{remark}\label{rem:why-not-four}
For $n=4$ one has
 $a_4=1,$ $ b_4=\frac12$ and $ \lambda_4=2,$
so $k=k_K$ and the Hessian correction in
\eqref{eq:higher-Ek} disappears.  The same computation gives
\[
 \cQ_g(k_K)
 =-\frac32\int_M|dV|^2\,dv_g
   +\frac12\int_M|F|^2\,dv_g
 =-\int_M|dV|^2\,dv_g
   +\frac12\int_MV|F|^2\,dv_g.
\]
This is not enough to deduce a sign.  This is precisely where four-dimensional selfduality enters: summing the identities
$|d(V\mp\chi)|^2=2V|F^\pm|^2$ gives
 $V|F|^2=|dV|^2+|d\chi|^2.$
Consequently, we recover the expression of \cref{thm:energy-identity}:
\[
 \cQ_g(k_K)
 =-\frac12\int_M\bigl(|dV|^2-|d\chi|^2\bigr)\,dv_g
 =-\frac12I_K.
\]
\end{remark}

We now deduce consequences for metrics carrying a parallel spinor, which are semistable by a noncompact version of \cite{DaiWangWei}.
\begin{proof}[Proof of \cref{cor:higher-special-holonomy}]
Dai-Wang-Wei provide a factorization of $2$-tensors in the presence of a parallel spinor in\cite[Lemma~2.4]{DaiWangWei}. The cutoff estimates of \eqref{eq:cutoff-spinor} give
\[
 0\leqslant\lim_{R\to+\infty}\cQ_g(\chi_Rk)=\cQ_g(k).
\]
Thus \eqref{eq:higher-energy-signed} forces $dV=0$, $F=0$, hence $K$ is parallel. In a universal cover, the metric is the product of a line and a Ricci-flat which is asymptotically Euclidean, hence Euclidean.
\end{proof}

\begin{remark}
The irreducible ALC $\mathrm G_2$ metrics of \cite{FHN21} do not satisfy \eqref{eq:higher-ALF-decay}, nor does it satisfy the ALC version consisting of replacing $g_{\mathbb R^{n-1}}$ by a Ricci-flat cones for which most of our analysis extends verbatim. A stronger similar result for ALC $G_2$ manifold is \cite[Theorem~C]{FHN26}.
\end{remark}

\subsection{Instability of Riemannian Myers-Perry and Schwarzschild-Tangherlini metrics}\label{sec:myers-perry}

The argument applies directly to the real Riemannian section of the Myers-Perry family \cite{MyersPerry,DowkerGauntlettGibbonsHorowitz}, which generalize the Riemannian version of Kerr metrics in dimension $4$. Its regular end is a flat quotient rather than literally the product in \eqref{eq:higher-ALF-decay}, but the potential $u$ is explicit, so $k$ can be computed explicitly and has the required decay. Consider the Myers-Perry metrics

\begin{equation}\label{eq:euclidean-myers-perry}
\begin{aligned}
g_{\mathrm{MP}}
={}&\tfrac{\Sigma}{D(r)}\,dr^2+\Sigma\,d\vartheta^2
+r^2\cos^2\vartheta\,g_{\mathbb S^{n-4}}+\tfrac{D(r)}{\Sigma}
 \bigl(d\tau-\alpha\sin^2\vartheta\,d\varphi\bigr)^2+\tfrac{\sin^2\vartheta}{\Sigma}
 \bigl((r^2-\alpha^2)d\varphi+\alpha\,d\tau\bigr)^2 .
\end{aligned}
\end{equation}
where
\[
 \Sigma=r^2-\alpha^2\cos^2\vartheta,\qquad
 D(r)=r^2-\alpha^2-mr^{5-n},\qquad
 \rho^2=r^2-\alpha^2\sin^2\vartheta,\qquad
 V=1-\frac{mr^{5-n}}{\Sigma}.
\]
In particular, for $K=\partial_\tau$, $|K|_g^2=g_{\mathrm{MP}}(\partial_\tau,\partial_\tau)=V.$ Let $r_0$ be the largest zero of $D$.  Then
\[
 u=-\frac{\rho^2}{2}
 +\frac{3(n-2)m}{4(n-1)}
  \int_{r_0}^r
  \frac{r_0^2-s^2}{s^{n-4}D(s)}\,ds
\]
satisfies $\Delta_gu=a_nV+\lambda_n$.  Indeed, if $w=u+\rho^2/2$, then from $\Delta_g(-\frac{\rho^2}{2})=(n-1)-2(1-V),$ we find the explicit
\[
 \Delta_gw=-\frac{(r^{n-4}D w')'}{\Sigma r^{n-4}}
            =(2-a_n)(1-V).
\]
which can be integrated and yields $\nabla^2w=\mathcal O(r^{3-n})$. Thus every
smooth nonflat real Riemannian Myers-Perry metric is linearly unstable. 

When $\alpha=0$, there is more freedom to obtain smooth metrics from \eqref{eq:euclidean-myers-perry}. Namely, for $(Y^{n-2},h)$ compact Einstein with $\Ric_h=(n-3)h$, then the full Schwarzschild-Tangherlini family of metrics, considered in \cite{GibbonsHartnoll,GibbonsHartnollPope}, is
\begin{equation}\label{eq: STeuc}
     g_{ST}=f\,d\tau^2+f^{-1}dr^2+r^2h,\qquad \text{with }
 f=1-\left(\frac{r_0}{r}\right)^{n-3},
\end{equation}
is smooth and Ricci-flat when $\tau$ is chosen to be $4\pi r_0/(n-3)$-periodic. The preceding function $u$ still satisfies $\Delta_{g_{ST}}u=a_nV+\lambda_n$ with $\alpha=0$ since the radial Laplacian is independent of $h$.  Hence the same tensor $k$ has the same decay, is TT and $\cQ_{g_{ST}}(k)<0$ for $n\geqslant5$.

\section{Identities on gravitational instantons and proof of the main theorem}
\label{sec:global-sign}

\subsection{Scalar potentials and main result}

Throughout this section, $K$ is a Killing field on the Ricci-flat four-manifold $(M,g)$, and $u$ and $\chi$ satisfy the equations
\begin{equation}\label{eq:intro-local-equations}
 \Delta u=V+2,\qquad \text{ and }\qquad d\chi=\iota_K(*F),
\end{equation} on $M$.  The existence of $\chi$ and $u$ on ALF manifolds is deferred to \cref{sec:end-analysis}.

\begin{proposition}[\cref{sec:end-analysis}]
\label{prop:analytic-input}
Let $(M^4,g)$ be a complete Ricci-flat ALF manifold satisfying
\eqref{eq:intro-end}, and let $K$ be a nonzero bounded Killing field.
After multiplying $K$ by a nonzero constant, for every
$0<\varepsilon<\frac12$ one has, on the end,
\[
 |\nabla_b^j(K-T)|_b+|\nabla_b^j(V-1)|_b
   =\mathcal O(r^{-1-j}),\qquad \text{ and }\qquad
 |\nabla_b^jF|_b=\mathcal O(r^{-2-j}),\qquad \text{for }j\geqslant0.
\]
There are global functions $u$ and $\chi$ respectively satisfying 
 $\Delta u=V+2$ and $d\chi=\iota_K(*F)$ and for every $j\geqslant0$,
\[
 |\nabla_b^j(u+r^2/2)|_b=\mathcal O(r^{1+\varepsilon-j}),
 \qquad
 |\nabla_b^j\chi|_b=\mathcal O(r^{-1+\varepsilon-j}).
\]
The symmetric $2$-tensor
 $k_K=\nabla^2u-\kappa\otimes\kappa+\frac{V+1}{2}g$ satisfies
 $$|\nabla^jk_K|=\mathcal O(r^{-1+\varepsilon-j})\qquad \text{ for }0\leqslant j\leqslant2.$$ In particular, all of the integrals used below converge: for $M_R=\{r<R\}$, we have as $R\to+\infty$
\begin{equation}\label{eq:standing-fluxes}
\begin{split}
 &\int_{\partial M_R}|k_K|\,|\nabla k_K|\,dA\longrightarrow0,\\
 &\int_{\partial M_R}
   |(F^+\circ F^-)(\nabla u,\nu)|\,dA\longrightarrow0,\\
 &\int_{\partial M_R}|1-V\pm\chi|\,
   |\partial_\nu(V\pm\chi)|\,dA\longrightarrow0.
\end{split}
\end{equation}
\end{proposition}
The proof is deferred to \cref{sec:end-analysis}.

\begin{lemma}\label{lem:scalar-equations}
With $\iota_K(*F)=d\chi$, on $\{K\neq0\}$, we have the following formulae
\begin{align}
 \iota_KF^\pm&=-\tfrac12d(V\mp\chi),
 \label{eq:chiral-contractions}\\
 |d(V\mp\chi)|^2&=2V|F^\pm|^2,\label{eq:chiral-norm-identity}\\
 \Delta(V\mp\chi)&=-2|F^\pm|^2\leqslant 0. \label{eq:Ernst}
\end{align}
\end{lemma}

\begin{proof}
Equation \eqref{eq:chiral-contractions} is
$\iota_KF^\pm=\tfrac12(\iota_KF\pm\iota_K(*F))
              =\tfrac12(-dV\pm d\chi)$.

For \eqref{eq:chiral-norm-identity}, work in a pointwise orthonormal
frame $(e_1,\dots,e_4)$ with $K=\sqrt V\,e_1$.  A basis of
$\Lambda^+$ is
\[
 \omega_1^+=e^1\wedge e^2+e^3\wedge e^4,\qquad
 \omega_2^+=e^1\wedge e^3-e^2\wedge e^4,\qquad
 \omega_3^+=e^1\wedge e^4+e^2\wedge e^3.
\]
and any $\alpha\in\Lambda^+$ writes $\alpha=a\omega_1^++b\omega_2^++c\omega_3^+$
with $|\alpha|^2=2(a^2+b^2+c^2)$.  Then we get
$\iota_K\alpha=\sqrt V(a\,e^2+b\,e^3+c\,e^4)$, hence
$|\iota_K\alpha|^2=V(a^2+b^2+c^2)=\tfrac V2|\alpha|^2$. The same
computation applies to $\Lambda^-$.  Applied to $F^\pm$ this gives
\eqref{eq:chiral-norm-identity}.

For \eqref{eq:Ernst}, combine $\Delta V=-|F|^2=-|F^+|^2-|F^-|^2$, by \eqref{eq:Delta-V},
with
\begin{align*}
 \Delta\chi
 &=-\nabla^j(K^i(*F)_{ij})
  =-(\nabla^jK^i)(*F)_{ij}-K^i(\delta(*F))_i\\
 &=-\tfrac12F^{ji}(*F)_{ij}-0
  =\tfrac12F^{ij}(*F)_{ij}
  =|F^+|^2-|F^-|^2,
\end{align*}
where we used $\delta(*F)=-*dF=0$ and, in the last step, that
$F^\pm_{ij}(*F)^{ij}=\pm F^\pm_{ij}F^{\pm ij}=\pm2|F^\pm|^2$ while the
mixed contractions vanish.  Subtracting gives \eqref{eq:Ernst}. 
\end{proof}

\begin{remark}\label{rem:AADSNS-conventions}
With the identifications of \cite[Def.~4.1, Prop.~4.3,
Eqs.~(4.7)-(4.13)]{AADSNS}, the functions $V-\chi$ and $V+\chi$ are
their two so-called Riemannian Ernst potentials. The use of
these potentials for ALF instantons, together with the maximum
principle below, appears in \cite[Sec.~4, Lemmas~5.1-5.2]{AADSNS},
building on \cite{GibbonsHawkingSymmetries}.
\end{remark}

We now assume the following \emph{normalization} on an end:
\begin{equation}\label{eq:normalization}
 V+\chi\to c,\qquad V-\chi\to c,
\end{equation}
where $c>0$ is the constant appearing in $\Delta u=V+2c$
(\cref{thm:scalar-response}). The existence of the constant $c$ will follow from
\cref{lem:bounded-K-asymptotics,lem:chi-existence}.

On the ends of interest, $c=1$ after
rescaling $K$ to make $\lim|K|=1$.  The maximum principle applied
to \eqref{eq:Ernst} and \eqref{eq:normalization} gives, since both
right-hand sides of \eqref{eq:Ernst} are nonpositive,
\begin{equation}\label{eq:Ernst-upper}
 V+\chi\leqslant 1,\qquad V-\chi\leqslant 1
 \qquad\text{on }M.
\end{equation}

We need one last identity before proving our main theorem. 
\begin{proposition}\label{prop:mixed-Ernst}
Under the hypotheses of \cref{thm:intro-bounded}, and with
$c=1$,
\begin{align}
 I_K
 &:=\int_M(|dV|^2-|d\chi|^2)\,dv_g\label{eq:I-def}\\
 &=2\int_M(1-V-\chi)|F^+|^2\,dv_g
 \label{eq:I-minus}\\
 &=2\int_M(1-V+\chi)|F^-|^2\,dv_g.
 \label{eq:I-plus}
\end{align}
In particular $I_K\geqslant0$, with equality if and only if $F^+=0$ or
$F^-=0$ on $M$.
\end{proposition}

\begin{proof}
Set $w_+=1-V-\chi$ and $w_-=1-V+\chi$; both are nonnegative by
\eqref{eq:Ernst-upper}.  On the exhaustion $M_R=\{r<R\}$, Green's
formula and \cref{lem:scalar-equations} give
\begin{align}
 \int_{M_R}\langle dw_+,dw_-\rangle\,dv_g
 &=\int_{M_R}w_+\Delta w_-\,dv_g
   +\int_{\partial M_R}w_+\,\partial_\nu w_-\,dA\notag\\
 &=\int_{M_R}w_+V^{-1}|d(V-\chi)|^2\,dv_g
   +\int_{\partial M_R}w_+\,\partial_\nu w_-\,dA,
 \label{eq:Green-plus}
\end{align}
where we used \eqref{eq:chiral-norm-identity} and
\eqref{eq:Ernst} to write $\Delta w_-=2|F^+|^2=V^{-1}|d(V-\chi)|^2$.
The left-hand side is
$\int_{M_R}(|dV|^2-|d\chi|^2)\,dv_g$. By the boundary
hypothesis \eqref{eq:standing-fluxes} the boundary integral is
$\mathcal{O}(R^{-1+2\varepsilon})\to0$.  Letting $R\to+\infty$ proves
\eqref{eq:I-minus}.  Interchanging $w_+$ and $w_-$ gives
\eqref{eq:I-plus}.

If $F^+$ and $F^-$ are both nonzero, then the strong maximum
principle applied to \eqref{eq:Ernst} gives $w_+>0$ and $w_->0$
strictly, so \eqref{eq:I-minus} is strictly positive.  Conversely,
if $F^+=0$,
then \eqref{eq:chiral-norm-identity} gives
$d(V-\chi)=0$, so $V-\chi\equiv1$ after normalization and $I_K=0$.
If $F^-=0$, similarly $V+\chi\equiv1$ and $I_K=0$.
\end{proof}

\begin{theorem}\label{thm:energy-identity}
Under the hypotheses of \cref{thm:intro-bounded}, with $k_K$ the
tensor of \cref{thm:scalar-response} with $c=1$, we have
\begin{equation}\label{eq:energy-I}
 \cQ_g(k_K)=-\frac12I_K\leqslant0,
\end{equation}
with strict inequality unless $F$ is self-dual or anti-self-dual.
\end{theorem}

\begin{proof}
On the exhaustion $M_R$, integration by parts gives
\begin{equation}\label{eq:energy-boundary}
 \int_{M_R}\langle\cE_gk_K,k_K\rangle\,dv_g
 =\int_{M_R}|\nabla k_K|^2\,dv_g
  -2\int_{M_R}\langle\mathring{\Rm}\,k_K,k_K\rangle\,dv_g
  +\text{boundary terms},
\end{equation}
the boundary term being $\mathcal{O}(R^{-1+2\varepsilon})$ by
the first limit in
\eqref{eq:standing-fluxes}.  On the other hand, using
$\cE_gk_K=2F^+\circ F^-$ and $\delta(F^+\circ F^-)=0$
from \cref{prop:harmonic-product}, since for every symmetric two-tensor $A$, 
$\langle A,\kappa\otimes\kappa\rangle=A(K,K)$ we find
\begin{align}
 \int_{M_R}\langle\cE_gk_K,k_K\rangle\,dv_g
 &=\int_{M_R}\langle 2F^+\circ F^-,\,\nabla^2u-\kappa\otimes\kappa+\tfrac{V+1}{2}g\rangle\notag\\
 &=-\int_{M_R}2(F^+\circ F^-)(K,K)\,dv_g
   +\text{ boundary terms}.\label{eq:energy-pairing}
\end{align}
Here, the trace term dropped out because $F^+\circ F^-$ is
trace-free, the Hessian term was integrated by parts against the
divergence-free $F^+\circ F^-$, and the boundary terms are $\int_{\partial M_R}
2(F^+\circ F^-)(\nabla u,\nu)\,dA=\mathcal{O}(R^{-1+2\varepsilon})$ by
the second limit in
\eqref{eq:standing-fluxes}.  From \eqref{eq:chiral-contractions},
\begin{align*}
 2(F^+\circ F^-)(K,K)
 &=2\langle\iota_KF^+,\iota_KF^-\rangle
  =\tfrac12\langle d(V-\chi),d(V+\chi)\rangle\\
 &=\tfrac12(|dV|^2-|d\chi|^2).
\end{align*}
Indeed, directly from the definition of $\circ$,
\[
 (F^+\circ F^-)(K,K)
 =K^iK^j(F^+)_i{}^p(F^-)_{jp}
 =(\iota_KF^+)^p(\iota_KF^-)_p
 =\langle\iota_KF^+,\iota_KF^-\rangle.
\]
Passing to the limit $R\to+\infty$ and combining with
\cref{prop:mixed-Ernst} proves \eqref{eq:energy-I}.
\end{proof}

The equality case finally identifies a local hyperkähler structure. 
\begin{proposition}[{\cite[Proof of Prop.~15]{BO}}] \label{prop:chiral-half-flat}
Let $K\not\equiv0$ be a Killing vector field on a Ricci-flat four-manifold.  If
$F^+=0$, then $W^+=0$ while if $F^-=0$, then $W^-=0$.  Consequently
$g$ is locally hyperk\"ahler.
\end{proposition}

The converse of \cref{thm:energy-identity} is also well-known, and
completes the equivalence in \eqref{eq:intro-dichotomy}.

\begin{proof}[Proof of \cref{thm:intro-bounded}]
The existence of $u$, $\chi$, and the decay estimates on $k_K$ are provided by \cref{sec:end-analysis} below.  If neither Weyl half vanishes, \cref{prop:chiral-half-flat} shows that neither $F^+$ nor $F^-$ vanishes, so \cref{thm:energy-identity} gives the finiteness and negativity of $\cQ_g(k_K)$, hence the instability of the metric.  If $W^+=0$ or $W^-=0$, one has gives $\cQ_g\geqslant0$, and $g$ is locally hyperk\"ahler by \cref{prop:chiral-half-flat}.  The nonlinear dynamical and variational statement on ALF ends is \cite[Prop.~5.3]{KO}, whose hypotheses are satisfied.
\end{proof}

\subsection{Application to the Li-Sun instantons}
\label{sec:applications}

Li and Sun construct, for every $\beta\in(0,1)$ and every integer
$n\geqslant1$, a complete toric Ricci-flat manifold $M_n$ with
$b_2(M_n)=n$ and an $\mathrm{AF}_\beta$ end \cite[Thm.~1.1]{LS}. For
$n\geqslant3$ these metrics are not locally Hermitian in either
orientation. The first explicit non-Hermitian representative of their family is described in coordinates in \cite{Teo}.

In a basis of the torus Lie algebra, the
asymptotic Gram matrix is
\begin{equation}\label{eq:LS-Gram}
 G_\infty=\begin{pmatrix}
   \rho^2&\beta\rho^2\\
   \beta\rho^2&\beta^2\rho^2+1
 \end{pmatrix},
\end{equation}
and the normalized bounded Killing field is
\begin{equation}\label{eq:LS-K}
 K=X_2-\beta X_1,\qquad V = |K|^2\to1.
\end{equation}
Note that for irrational $\beta$ the orbits of $K$ are dense in the torus.  The manifolds $M_n$ are simply connected, so the primitive $\chi$ of $\iota_K(*F)$ exists globally.  By \cite[Proposition~A.1]{LS}, the Li-Sun metrics admit a $T^2$-invariant chart at infinity in which \eqref{eq:intro-end} holds with decay rate $\tau=1$, so \cref{thm:intro-bounded} applies directly.

\section{Remarks on the limitation to ALF manifolds}

\subsection{Noncollapsed ALE ends}
\label{sec:ALE-remark}

The boundedness of the Killing vector field is essential in our proof.  On a nonflat ALE manifold, a nontrivial Killing field is asymptotic to a rotation, i.e.
$$K=Ax+\mathcal{O}(r^{-3})$$ with $A\in\mathfrak{so}(4)$, so $V=|K|^2=\mathcal{O}(r^2)$.
Solving $\Delta u=V+2c$ therefore forces $u=\mathcal{O}(r^4)$ and $k_K=\mathcal{O}(r^2)$,
which is not an admissible variation of the fixed asymptotic class.

Even in the exceptional cases for which $2F^+\circ F^-=\mathcal{O}(r^{-4})$ is already $L^2$, the equation $\cE_gq=2F^+\circ F^-$ is typically not solvable: it is obstructed against the $L^2$ cokernel element of $\mathcal{E}$ given by rescaling and reparametrization \cite{BiquardHein,OzuchIntegrability}.  

\subsection{Other asymptotics of Ricci-flat metrics}\label{sec:other-ends}

\begin{lemma}\label{lem:no-bounded-killing-quadratic-growth}
Let $(M^4,g)$ be a connected, complete, smooth Ricci-flat
Riemannian manifold. Suppose that, for some
$p\in M$ and all sufficiently large $R$, one has
\[
    \operatorname{Vol}_g B_g(p,R) \leqslant C R^2.
\]
Then every bounded Killing vector field on $(M,g)$ is parallel.
Consequently, if $(M,g)$ is nonflat, it admits no nonzero bounded
Killing vector field.
\end{lemma}

\begin{proof}
Let $K$ be a Killing vector field and set $V=|K|_g^2$. The
 Bochner identity for a Killing vector field is
\[
    -\frac{1}{2}\Delta |K|^2
       = |\nabla K|^2-\operatorname{Ric}(K,K)
       = |\nabla K|^2 \geqslant 0
\]
Thus, if $K$ is bounded, then $V$ is a bounded subharmonic function.

Writing $V_p(R)=\operatorname{Vol}_g B_g(p,R)$, the volume assumption
gives
\[
    \int^\infty \frac{R\,dR}{V_p(R)}
       \geqslant \frac{1}{C}\int^\infty \frac{dR}{R}
       =\infty.
\]
By Grigor'yan's volume criterion \cite[Theorem~7.3 and Example~7.1]{Grigoryan}, $(M,g)$ is
parabolic. On a parabolic manifold every bounded subharmonic
function is constant
\cite[Theorem~5.1(3)]{Grigoryan}. Hence $|K|^2$ is constant,
and the Bochner identity implies $\nabla K=0$.

If $K\not\equiv0$, this implies that the universal cover splits as
\[
    (\widetilde M,\widetilde g)
       = (\mathbb R,dt^2)\times (N^3,h),
       \qquad \operatorname{Ric}_h=0.
\]
Thus $\widetilde g$, and hence $g$, is flat.
\end{proof}

For the standard asymptotic models one has
\[
  \operatorname{Vol}_g B_g(p,R)\approx
  \begin{cases}
      R^2,     & \mathrm{ALG}\ \text{or}\ \mathrm{ALG}^* \ \text{or Kasner},\\
      R,       & \mathrm{ALH},\\
      R^{4/3}, & \mathrm{ALH}^*,
  \end{cases}
\]
so our starting assumption cannot be satisfied by the other above typical asymptotics.

\appendix

\section{Analysis on ALF manifolds}
\label{sec:end-analysis}

We use the definition and notation of
\cite[Definition~1.1]{BiquardGauduchon}.  Thus the end is diffeomorphic
to $(A,\infty)\times L$, where $L$ is the total space of an
$S^1$-bundle over $S^2$, or a finite quotient thereof.  On $L$ there
are a one-form $\eta$, a vector field $T$, and a $T$-invariant metric
$\gamma$ on $\ker\eta$ such that
\[
 \eta(T)=1,\qquad \iota_Td\eta=0,
\]
and the transverse metric $\gamma$ has constant curvature $+1$.
Extend $\gamma$ by $\gamma(T,\cdot)=0$ and set
\[
 b=dr^2+r^2\gamma+\eta^2.
\]
Then we say that a metric is ALF if it satisfies
\begin{equation}\label{eq:decay ALF}
    g=b+h \qquad \text{ with} \qquad|\nabla_b^jh|_b\leqslant C_j r^{-1-j}\qquad \text{ for }  r \geqslant R>0
\end{equation} 
for every $j\geqslant0$.
Cartan's formula gives $\mathcal L_T\eta=0$, hence $T$ is a unit
Killing field of $b$.  The orbits of $T$ are not required to be closed.

We first prove that the bounded Killing vector field $K$ is necessarily asymptotic to a multiple of the asymptotic vector field $T$.

\begin{lemma}\label{lem:bounded-K-asymptotics}
Let $g$ be nonflat and suppose $K$ is a nonzero bounded Killing
field.  After multiplying $K$ by a nonzero constant,
\begin{equation}\label{eq:bounded-K-expansion}
 |\nabla_b^j(K-T)|_b=\mathcal O(r^{-1-j}),\qquad j\geqslant0.
\end{equation}
In particular,
$|\nabla_b^j(V-1)|_b=\mathcal O(r^{-1-j})$ for $j\geqslant0$.

\end{lemma}

\begin{proof}

The proof of \cite[Proposition~B.2]{AADSNS} through
equation~(B.17) uses only the boundedness of $K$ and the ALF estimates,
not the assumption that its orbits are closed.  In the notation above
it gives
\begin{equation}\label{eq:expKilling}
    |K-cT-Y|\leqslant C r^{-1},\qquad\text{ and }\qquad |\nabla K|\leqslant C r^{-2} \qquad \text{ for } r \geqslant R>0
\end{equation}
where $Y$ is a parallel translation in the Euclidean component $dr^2 + r^2\gamma$ of the model at infinity.

We claim that $Y=0$, otherwise the flow of $K$ would translate the nonflat core of the manifold to infinity, where curvature decays to zero.  Choose $p\in M$ with
$|\Rm_g|(p)>0$ and set $a=\frac12|\Rm_g|(p)$.  Since
$|\Rm_g|\to0$ at infinity, the set
\[
 \mathcal C_a:=\{q\in M:|\Rm_g|(q)\geqslant a\}
\]
is nonempty and compact.  If $\Phi_t$ is the flow of $K$, then
$\Phi_t(\mathcal C_a)=\mathcal C_a$, since $\Phi_t$ is an isometry.
In particular, for any $q\in M$,
\[
 d_g\bigl(\Phi_t(q),\mathcal C_a\bigr)
 =d_g(q,\mathcal C_a)
\]
for every $q\in M$ and $t\in\mathbb R$.  Thus no orbit of $K$ can escape to infinity. 

Suppose that $Y\neq0$ and suppose that $q = A\frac Y{|Y|}$ is in the ALF chart at infinity, i.e. $r(q) = A>R$ to be chosen large enough. Then, since, up to an error of the size of the fiber, $r$ is the distance function on the model space, there is $C'>0$ depending on the geometry of $(M,g)$ but independent on $q$ such that
\begin{equation}\label{eq:radius and distance to Ca}
    r(\Phi_t(q))-C'\leqslant d_g\bigl(\Phi_t(q),\mathcal C_a\bigr)
 =d_g(q,\mathcal C_a)\leqslant  r(q)+C'.
\end{equation}
Now, from \eqref{eq:expKilling} and \eqref{eq:decay ALF} we see that when $r(\Phi_t(q)) > 100 (C_1+C+C')$, then $\frac{d}{dt}r(\Phi_t(q))\geqslant \frac{|Y|}{2}$, hence $r(\Phi_t(q)) \to +\infty$ as $t\to+\infty$ which contradicts \eqref{eq:radius and distance to Ca}. So $Y = 0$ and $K = c T + \mathcal{O}(r^{-1})$.

If $c=0$, then $V=|K|^2\to0$.  Since
$\Delta V=-|F|^2\leqslant0$, the maximum principle on $\{r\leqslant R\}$
followed by $R\to+\infty$ gives $V\equiv0$, which is a contradiction.
Thus $c\ne0$, and up to rescaling $K$ we may assume that $c=1$.  The derivative estimates in the proof of \cite[Proposition~B.2]{AADSNS} (which does not use the fact that the orbits of $K$ are closed), give \eqref{eq:bounded-K-expansion}.
\end{proof}

The rest of this section verifies the existence of $u$ and $\chi$ and their decay estimates.

\begin{lemma}\label{lem:scalar-end}
Assume the geometric hypotheses of \cref{thm:intro-bounded} and
normalize $K$ as above.  For every $0<\varepsilon<\tfrac12$ there is
a solution of
\begin{equation}\label{eq:u-Poisson}
 \Delta u=V+2, \text{ with}
\end{equation}
\begin{equation}\label{eq:u-asymptotics}
 u=-\tfrac{r^2}{2}+w,\qquad
 |\nabla_b^jw|=\mathcal{O}(r^{1+\varepsilon-j}),
\end{equation}
and the tensor $k_K$ of \eqref{eq:k-def} satisfies
\begin{equation}\label{eq:end-k-decay}
 |\nabla^jk_K|=\mathcal{O}(r^{-1+\varepsilon-j}),\qquad j=0,1,2.
\end{equation}
In particular, it satisfies the decay assumptions for dynamical destabilizations of ALF metrics from \cite[Definition 5.1]{KO}.
\end{lemma}

\begin{proof}
On either model, $\Delta_b(-r^2/2)=3$, and
for every ALF model metric,
\begin{equation}\label{eq:model-cancellation}
 \nabla_b^2(-r^2/2)-\eta\otimes\eta+b=0.
\end{equation}
Combined with the
order-one hypothesis on $g$ and \eqref{eq:bounded-K-expansion}, the
error
\[
 f:=(V+2)-\Delta_g(-r^2/2)
\]
lies in the space $C^{0,\alpha}_{1-\varepsilon}$ of tensors decaying like
$r^{-1+\varepsilon}$.

Since $-1-\varepsilon$ is not an exceptional weight,
\cite[Theorem~2.25 and Corollary~2.28]{KO} shows that
$\Delta_g:
 C^{2,\alpha}_{-1-\varepsilon}(M)\longrightarrow C^{0,\alpha}_{1-\varepsilon}(M)$ is Fredholm.  Moreover,
\cite[(2.26)]{KO} identifies its cokernel with the kernel of
 $\Delta_g:C^{2,\alpha}_{2+\varepsilon}(M)
 \longrightarrow C^{0,\alpha}_{4+\varepsilon}(M).$

Functions in
$C^{2,\alpha}_{2+\varepsilon}$ are
$\mathcal O(r^{-2-\varepsilon})$, tend to zero at infinity, hence vanish by maximum principle.  Consequently, the operator
\[
 \Delta_g:C^{2,\alpha}_{-1-\varepsilon}(M)
 \longrightarrow C^{0,\alpha}_{1-\varepsilon}(M)
\]
is surjective, and one can solve $\Delta_gw=f$ with
 $w\in C^{2,\alpha}_{-1-\varepsilon}(M)$ uniquely, modulo the finite-dimensional kernel of $\Delta_g$. This kernel is nontrivial since it contains constants, and possibly asymptotically linear functions. The different choices of $u$ will not modify our argument since Hessians of harmonic functions lie in $\ker\mathcal{E}$ (see \eqref{eq:E-natural}) and are integrated by parts against a divergence-free term in the computation of the quadratic term in \eqref{eq:energy-pairing}. 

Combining \eqref{eq:u-asymptotics} and the identity
\eqref{eq:model-cancellation}
gives \eqref{eq:end-k-decay}. 
\end{proof}

\begin{remark}\label{rem:no-torus}
If $g$ is toric, one can average the solution $u$ over the compact torus to obtain an invariant $u$, this simplifies the analysis on AF$_\beta$ manifolds for instance but is not used here. 
\end{remark}

\begin{lemma}\label{lem:chi-existence}
Under the same hypotheses, there is a smooth global function
$\chi$ on $M$ with $d\chi=\iota_K(*F)$ and
\begin{equation}\label{eq:chi-asymptotics}
 |\nabla_b^j\chi|=\mathcal{O}(r^{-1+\varepsilon-j}),\qquad \text{ for }j\geqslant0,
\end{equation}
which gives $|\nabla_b^j(V\pm\chi-1)|_b
 =O(r^{-1+\varepsilon-j}),$ for $ j\geqslant0$ at infinity. 
\end{lemma}

\begin{proof}
The one-form $\iota_K(*F)$ is closed (Cartan) and satisfies
$|\iota_K(*F)|=\mathcal{O}(r^{-2+\varepsilon})$. We find $\chi$ by solving
$\Delta_g\chi=\delta_g(\iota_K(*F))$ in $C^{2,\alpha}_{1-\varepsilon}$ using the same Fredholm theory as in the previous proof. Then the $1$-form $\gamma= \iota_K(*F)-d\chi$ is harmonic, coclosed, and $L^2$, which forces it to vanish by \cite[Lemma~1.7]{BiquardGauduchon}. 
\end{proof}

Combining the estimates gives
\begin{equation}\label{eq:end-basic-decay}
 V=1+\mathcal{O}(r^{-1+\varepsilon}),\qquad
 F=\mathcal{O}(r^{-2+\varepsilon}),\quad \text{ and } \quad
 \chi=\mathcal{O}(r^{-1+\varepsilon}),
\end{equation}
and the three boundary terms appearing in
\cref{thm:energy-identity,prop:mixed-Ernst},
\begin{equation}\label{eq:end-fluxes}
 \int_{\partial M_R}|k_K|\,|\nabla k_K|\,dA,\quad
 \int_{\partial M_R}|(F^+\circ F^-)(\nabla u,\nu)|\,dA,\quad
 \int_{\partial M_R}|1-V\pm\chi|\,|\partial_\nu(V\pm\chi)|\,dA,
\end{equation}
are all $\mathcal{O}(R^{-1+2\varepsilon})\to0$ for
$\varepsilon\in(0,\tfrac12)$.

\begin{remark}
\label{subsec:higher-ALF-analysis}
We finally record the higher dimensional controls from \cite{KO}. Let $\tau$ be the rate in \eqref{eq:higher-ALF-decay} and put
$u_0=-r^2/2$.  The identities
$a_n+\lambda_n=n-1$ and
$b_n+\lambda_n/n=1$ give
$a_nV+\lambda_n-\Delta_gu_0=\mathcal O(r^{-\tau})$.
For
\[
 \frac{n-3}{2}<\mu<\min\{\tau,n-1\},
 \qquad \text{ and }\qquad \mu-2\ \text{nonexceptional},
\]
the weighted theory of \cite[Corollary~2.28]{KO} gives $u=u_0+w$ with
\[
 \nabla_b^jw=\mathcal O(r^{2-\mu-j}),\qquad \text{ and }\qquad 
 \nabla_b^jk=\mathcal O(r^{-\mu-j}),\qquad 0\leqslant j\leqslant2.
\]
Since $\operatorname{Area}(\partial M_R)=\mathcal O(R^{n-2})$, the cutoff
and boundary errors are
\[
 \mathcal O(R^{n-3-2\mu})
 +\mathcal O(R^{n-3-\mu-\tau})
 +\mathcal O(R^{n-3-2\tau})=o(1).
\]
Thus $\cQ_g(k)$ is finite and
\begin{equation}
   \cQ_g(\chi_Rk)\xlongrightarrow[R\to+\infty]{} \cQ_g(k).\label{eq:cutoff-spinor}
\end{equation}
\end{remark}

\end{document}